\documentclass[10pt]{amsart}

\usepackage{verbatim}
\usepackage{eucal,url,amssymb,stmaryrd,
enumerate,amscd,
}

\usepackage[hypertexnames=false,
backref=page,
    pdftex,
    pdfpagemode=UseNone,
    breaklinks=true,
    extension=pdf,
    colorlinks=true,
    linkcolor=blue,
    citecolor=blue,
    urlcolor=blue,
]{hyperref}

\usepackage{amsfonts}
\usepackage{amsmath,amsthm,amssymb,amscd,enumerate,eucal,url,stmaryrd}

\numberwithin{equation}{section}

\newtheorem{thrm}{Theorem}[section]

\newtheorem{prop}[thrm]{Proposition}

\newtheorem{dfn}[thrm]{Definition}

\newtheorem{rmrk}[thrm]{Remark}

\def\frh{{\frak h}}

\def\Real{{\frak R}{\frak e}\,}
\def\Imag{{\frak I}{\frak m}\,}
\def\nilm{\Gamma\backslash G}

\def\zzz{{\!\!\!}}

\begin{document}

\begin{abstract}
We show that the property of existence of solution to
the Strominger system in dimension six is neither open nor closed under holomorphic deformations of the complex structure.
These results are obtained both in the case of positive slope parameter as well as in the case of negative
slope parameter in the anomaly cancellation equation.
\end{abstract}

\title[On the Strominger system and holomorphic deformations
]
{On the Strominger system and holomorphic deformations
}
\date{\today}

\author{Stefan Ivanov}
\address[S. Ivanov]{University of Sofia "St. Kl. Ohridski"\\
Faculty of Mathematics and Informatics\\
Blvd. James Bourchier 5\\
1164 Sofia, Bulgaria} \email{ivanovsp@fmi.uni-sofia.bg}

\author{Luis Ugarte}
\address[L. Ugarte]{Departamento de Matem\'aticas\,-\,I.U.M.A.\\
Universidad de Zaragoza\\
Campus Plaza San Francisco\\
50009 Zaragoza, Spain}
\email{ugarte@unizar.es}

\maketitle

\setcounter{tocdepth}{2} \tableofcontents

\section{Introduction}

\noindent
Let $X$ be a compact K\"ahler manifold with holomorphically trivial canonical bundle.
By Yau's solution to the Calabi conjecture \cite{Yau78}, $X$ admits a Ricci-flat K\"ahler metric.
However, there are many compact complex manifolds $X$ with holomorphically trivial canonical bundle
not admitting any K\"ahler metric, and it arises the question of existence of canonical Hermitian metrics
in this setting. Candidates for dealing with this problem are the solutions
to the Strominger system.
In~\cite{Str} Strominger investigated the heterotic superstring
background with non-zero torsion, which led to a complicated system of partial differential equations.
Roughly speaking, the system requires $X$ to be a compact
complex
conformally balanced manifold with holomorphically trivial canonical bundle, equipped with an instanton $A$
satisfying the anomaly cancellation equation
\begin{equation}\label{anomaly-canc}
dT=\frac{\alpha'}{4}\left({\rm tr}\, \Omega\wedge\Omega -{\rm tr}\, \Omega^A\wedge \Omega^A
\right),
\end{equation}
for a non-zero constant $\alpha'$ which is the slope parameter in string theory.
Here $T$ denotes the 3-form defined by the torsion of the Strominger-Bismut connection
of the Hermitian metric,
$\Omega$ is the curvature form of some metric connection
$\nabla$ and $\Omega^A$ is the curvature form of the instanton $A$,
i.e. $\Omega^{A}$ is the curvature form of a connection on a holomorphic hermitian vector bundle $(E,h)$ over $X$ which satisfies the Hermitian-Yang-Mills equation.

Any K\"ahler Calabi-Yau manifold solves the Strominger system taking e.g. $\Omega^A=\Omega$, and in the last years many efforts have been made in
constructing solutions on non-K\"ahler manifolds, more specifically in complex dimension 3.
In \cite{Li-Yau} Li and Yau obtained the first non-K\"ahler solutions to the Strominger system on a K\"ahler Calabi-Yau manifold,
which was
further extended in \cite{A-Gar}.
Based on a construction in
\cite{GP}, Fu and Yau proved the existence of solutions to the
Strominger system on non-K\"ahler Calabi-Yau manifolds given as a
$\mathbb{T}^2$-bundle over a $K3$ surface~\cite{Fu-Yau-alpha-positive,Fu-Yau-alpha-negative}, see \cite{PPZ} for solutions in all dimensions. Recently, Phong, Picard and Zhang investigating the anomaly flow in \cite{PPZ1} recapture the Fu-Yau results. Explicit smooth solutions to the Strominger system on
compact non-K\"ahler Calabi-Yau 3-folds with infinitely many topological types and sets of Hodge numbers are constructed very recently in \cite{FHP} relying on the  generalized Calabi-Gray construction outlined in \cite{Fei}.
Explicit solutions on nilmanifolds have been constructed in \cite{CCDLMZ,FIUV09,FIUV2014,UV2},
and more recently invariant solutions were also found on some solvmanifolds and on compact quotient
of SL(2,$\mathbb{C}$) in \cite{Fei-Yau} and \cite{OUV1}, see also \cite{AGarcia}.
On the other hand, some recent developments in the infinitesimal moduli problem of the Strominger system were presented in \cite{OsSv1,OsSv2}
and the relation of the Strominger system with generalized geometry is obtained for instance in
\cite{Anderson-G-S} and \cite{Gar-Rubio-Tipler}, see the notes \cite{MGarcia} for more results on the subject.

Solutions of the Strominger system under deformations of the SU(3) structure, i.e. holomorphic deformations of the complex structure together with a deformation of the conformally balanced condition is investigated by de la Ossa and Svanes in \cite{Ossa-Svanes}, see also \cite{BT} where the authors consider deformations of the complex
structure only. It is shown in \cite{Ossa-Svanes}  that under deformations of the complex
structure alone the conformally balanced condition gives rise to a complicated PDE which seems to be a difficult equation to satisfy. In this note we confirm in some sense this difficulty.

We say that a compact complex manifold $X$ has the \emph{Strominger property} if $X$ admits a solution to the Strominger system
(see Definition~\ref{Strominger-property}).
Any K\"ahler Calabi-Yau manifold has the Strominger property, and any $X$ satisfying the Strominger property is in particular a balanced
Hermitian manifold.
By the well-known Bogomolov-Tian-Todorov theorem (announced in \cite{Bogomolov} and proved independently in \cite{Tian,Todorov}),
the deformation space of any K\"ahler Calabi-Yau manifold is unobstructed. In contrast to this case,
where any sufficiently small deformation of a K\"ahler Calabi-Yau manifold
is again K\"ahler Calabi-Yau, it was proved in \cite{AB} by Alessandrini and Bassanelli,
relying on Michelsohn's investigations in \cite{Mi},
that the balanced property is not stable under small deformations of the complex structure.
Since the compact complex manifolds $X$ satisfying the Strominger property constitute an important class
between the class of K\"ahler Calabi-Yau manifolds and the class of balanced manifolds, it arises the natural question
of whether that class is stable or not under holomorphic deformations of the complex structure.

Our first goal in this paper is to prove that the Strominger property is not stable.
We will distinguish two cases depending on the
sign of the slope parameter $\alpha'$ in \eqref{anomaly-canc}.
Furthermore, we will also prove that the Strominger property is not closed, i.e.
the central limit of a holomorphic family of solutions of the Strominger system may not
admit any solution.

\smallskip

The paper is structured as follows. In Section~\ref{sec-general} we recall some terminology and results on the
Strominger system.
In Section~\ref{sec-open} we construct a holomorphic family $(X_\mathbf{t})_{\mathbf{t} \in \Delta}$
such that $X_0$ admits a solution to the Strominger system
with positive $\alpha'$, but
there does not exist any balanced metric on $X_\mathbf{t}$ for arbitrary small values of $\mathbf{t}\not=0$
(see Proposition~\ref{X0-Strominger} and Theorem~\ref{Strominger-no-open}).
Hence, the Strominger property with positive slope parameter $\alpha'$ is not stable.
For the construction of such a family we first analyze the space of invariant complex structures on the
manifold underlying the Iwasawa manifold,
which leads us to consider a compact complex manifold $X_0$ satisfying certain specific condition on its
Bott-Chern cohomology group of bidegree~(2,2).

In Section~\ref{sec-closed} we prove that the Strominger property with positive slope parameter $\alpha'$
is not closed under holomorphic deformations. Indeed,
there is a holomorphic family of compact complex manifolds $(X_\mathbf{t})_{\mathbf{t}\in \Delta}$
such that $X_\mathbf{t}$ has a solution to the Strominger system with $\alpha'>0$ for any $\mathbf{t} \in \Delta-\{0\}$,
but there does not exist any such solution on its central limit $X_0$ (see Theorem~\ref{no-closed}).
To prove this result, we further study the family $(X_\mathbf{t})_{\mathbf{t} \in \Delta}$ constructed in \cite{COUV} to answer a question in \cite{Pop2}
about the closedness of the balanced property. We prove that for any $\mathbf{t}\not=0$,
the manifold $X_\mathbf{t}$ has a solution to the Strominger system with $\alpha'>0$. Hence, the non-closedness
result comes from the fact that the central limit $X_0$ does not admit any balanced metric.

Section~\ref{sec-negative} is devoted to the negative slope parameter case, and we show that again
the Strominger property with $\alpha'<0$ is neither open nor closed under holomorphic deformations.
It should be remarked that the results obtained in
Sections~\ref{sec-open} and~\ref{sec-closed} for $\alpha'>0$ cannot be applied directly to the
$\alpha'<0$ case. However, we make use of a non-flat instanton $A$ constructed in \cite{CCDLMZ} to reverse the sign
of the slope parameter in the solutions found in Sections~\ref{sec-open} and~\ref{sec-closed}.

\section{The Strominger system}\label{sec-general}

\noindent
The Strominger system \cite{Str,Hul} is a system of partial differential equations characterizing the compactification
of heterotic superstrings with torsion. The bosonic fields of the ten-dimensional supergravity which arises as low
energy effective theory of the heterotic string are the spacetime metric $g$,
the NS three-form field strength (flux) $H$, the dilaton $\phi$ and the
gauge connection $A$ with curvature 2-form $\Omega^A$. The bosonic geometry is of
the form $\mathbb{R}^{1,9-d}\times X^d$, where the bosonic fields are
non-trivial only on $X^d$, $d\leq 8$. We consider the  connection
$\nabla^+=\nabla^g + \frac12 H$,
where $\nabla^g$ is the Levi-Civita connection of the Riemannian metric~$g$.
This connection preserves the metric, $\nabla^+g=0$, and has totally
skew-symmetric torsion $H$.

A heterotic geometry preserves supersymmetry if and only if in ten
dimensions there exists at least one Majorana-Weyl spinor $\epsilon$ such
that the following Killing-spinor equations hold \cite{Str,Berg}
\begin{equation} \label{sup1}
\nabla^+\epsilon=0, \quad (d\phi-\frac12H)\cdot\epsilon=0, \quad \Omega^A\cdot\epsilon=0,
\end{equation}
where $\cdot$ means Clifford action of forms on
spinors.

The system of Killing spinor equations \eqref{sup1} together with the
anomaly cancellation condition \eqref{anomaly-canc} is known as the \emph{%
Strominger system} \cite{Str}.

In dimension $d=6$, Strominger shows in \cite{Str} that the first two equations in \eqref{sup1} imply the 6-manifold $X$ is a compact conformally balanced Hermitian manifold with holomorphically trivial canonical bundle. Indeed, the first equation means that the holonomy group of the connection $\nabla^+$ is contained in SU(3), i.e. the manifold $X$ admits a $\nabla^+$-parallel almost Hermitian structure $(g,J)$ with a nowhere vanishing $\nabla^+$-parallel $(3,0)$-form $\Psi_0$, and the second equation forces the almost complex structure $J$ to be integrable, the three form $H=-dF(J.,J.,J.)=JdF$, where $F$ is the K\"ahler 2-form, $F=g(.,J.)$, and the Lee form $\theta$ to be exact, $\theta=\delta F\circ J=2d\phi$.
Moreover, it is easy to see that the (3,0)-form $\Psi=e^{-2\phi}\Psi_0$ is closed, and therefore holomorphic, and the metric $\bar g=e^{-2\phi}g$ is balanced, i.e. its Lee form vanishes which is equivalent to the K\"ahler form $\bar F$ to be co-closed, $\delta\bar F=0$
(see \cite{Mi} for properties and  examples of balanced Hermitian manifolds).

The last equation in \eqref{sup1} is the
instanton condition which means that the curvature $\Omega^A$ is contained in the
Lie algebra $\mathfrak{s}\mathfrak{u}(3)$ of the Lie group SU(3), i.e. $\Omega^A\wedge F^2=0$ and
$(\Omega^A)^{0,2}=(\Omega^A)^{2,0}=0$, which is the well-known Donaldson-Uhlenbeck-Yau instanton.

\smallskip

Thus, let $X$ be a (compact) complex 3-fold with holomorphically trivial canonical bundle and fix a nowhere vanishing holomorphic (3,0)-form $\Psi$. Let $F$ be a Hermitian metric on $X$ and denote by $||\Psi||_F$ the norm of $\Psi$ with respect to the metric $F$. Li and Yau showed in \cite{Li-Yau} that the first two equations in \eqref{sup1} are equivalent to single equation $d(||\Psi||_F \cdot F^2)=0$.
Let $(E,h, A)$ be a holomorphic hermitian vector bundle  over $X$   with curvature 2-form $\Omega^A$.
According to \cite{Li-Yau}, the Strominger system can be written as follows
\begin{equation}\label{sup2}
\begin{aligned}
&d(||\Psi||_F \cdot F^2)=0, \\
& \Omega^A\wedge F^2=0,\quad (\Omega^A)^{0,2}=(\Omega^A)^{2,0}=0,\\
& d(JdF)=\frac{\alpha'}{4}\left({\rm tr}\, \Omega\wedge\Omega -{\rm tr}\, \Omega^A\wedge \Omega^A
\right).
\end{aligned}
\end{equation}

The connection $\nabla^+$ with totally skew-symmetric torsion preserving the Hermitian structure $(g,J,F)$ is unique determined by the torsion 3-form $T=H=JdF=-dF(J.,J.,J.)$. This connection was used by Bismut to prove a local index formula for the Dolbeault
operator on non-K\"ahler Hermitian manifolds  \cite{Bis} which was applied
in string theory \cite{BBDGE}. This remarkable connection is known as the Strominger-Bismut connection.

\smallskip

We introduce the following definition.

\begin{dfn}\label{Strominger-property}
We will say that a compact complex manifold $X$ has the \emph{Strominger property}
if $X$ admits a solution to \eqref{sup2}, that is to say,
$X$ has holomorphically trivial canonical bundle admitting a conformally balanced Hermitian metric and an instanton bundle $E$
solving the anomaly cancellation equation~\eqref{anomaly-canc}.
\end{dfn}

In this definition we refer to any solution of the Strominger system, that is, we do not restrict the connection or the instanton in the anomaly
cancelation equation.
However, for the solutions that we will construct in the paper it will be enough to
consider the Strominger-Bismut connection, i.e. $\Omega$ in~\eqref{anomaly-canc} will be the curvature of $\nabla=\nabla^+$.
The trivial instanton $A=0$ will be enough in our study of
the Strominger property for the case $\alpha'>0$
(but not for negative $\alpha'$, see Section~\ref{sec-negative}).

\smallskip

Our aim in this paper is to study the Strominger property under holomorphic deformations of the complex structure.
We will use the following terminology.
Let $\Delta$ be an open disc, that we can suppose around the origin, in $\mathbb{C}$.
Following \cite[Definition 1.12]{Pop2},
a given property $\mathcal{P}$ of a compact complex manifold is said to be \emph{open} (or \emph{stable})
under holomorphic deformations if for every holomorphic family of compact complex manifolds
$(X_\mathbf{t})_{\mathbf{t}\in \Delta}$ and for every $\mathbf{t}_0\in \Delta$ the following implication holds:

\medskip

$X_{\mathbf{t}_0}$ has the property $\mathcal{P}$ $\Longrightarrow$ $X_{\mathbf{t}}$ has the property $\mathcal{P}$
for all $\mathbf{t}\in\Delta$ sufficiently close to $\mathbf{t}_0$.

\medskip

A given property $\mathcal{P}$ of a compact complex manifold is said to be \emph{closed}
under holomorphic deformations,
if for every holomorphic family of compact complex manifolds $(X_\mathbf{t})_{\mathbf{t}\in \Delta}$
and for every $\mathbf{t}_0\in \Delta$ the following implication holds:

\medskip

$X_{\mathbf{t}}$ has the property $\mathcal{P}$ for all $\mathbf{t}\in \Delta - \{\mathbf{t}_0\}$ $\Longrightarrow$ $X_{\mathbf{t}_0}$
has the property $\mathcal{P}$.

\medskip

In~\cite{AB} and~\cite{COUV} it is proved, respectively, that the balanced property is neither open nor closed
(in the balanced class).
Roughly speaking, what we will prove in the next sections is that the balanced property is neither open nor closed
even in the Strominger class, both in the case of positive slope parameter as well as in the case of negative
slope parameter. This clearly implies the non-openness and the non-closedness of the Strominger property
under holomorphic deformations.

\section{The Strominger property with positive $\alpha'$ is not stable}\label{sec-open}

\noindent In this section we show that the Strominger property with positive slope parameter $\alpha'$ in \eqref{anomaly-canc}
is not stable by small deformations of the complex structure.
We will prove a stronger result, namely that a small holomorphic deformation of a compact complex manifold satisfying
the Strominger system with $\alpha'>0$
may not admit any balanced metric.

We will focus on the complex geometry of the (real) manifold underlying the Iwasawa manifold.
Let us denote by $M$ this 6-dimensional compact manifold.
It is well known that $M$ is a nilmanifold, i.e. it is a compact manifold $M=\nilm$ obtained as a quotient of a simply-connected
nilpotent Lie group $G$ by a lattice $\Gamma$ of maximal rank.
The Lie algebra of $G$ is isomorphic to $\frh_5=(0,0,0,0,13+42,14+23)$ (in the notation of \cite{Sal}).

We consider \emph{invariant} complex structures $J$ on $M$, that is to say, $J$ is a complex structure
coming from a left-invariant complex structure on $G$ by passing to the quotient.
By \cite{Sal} the canonical bundle of any $(M,J)$ is holomorphically trivial.
To our knowledge, it is an open question if the Iwasawa manifold, i.e. $M$ with its standard complex parallelizable structure,
can admit or not any solution to the Strominger system with $\alpha'>0$ (see Section~\ref{sec-negative} for more details).
Next we look for appropriate invariant complex structures on $M$.

There are many invariant complex structures on $M$, but in order to prove the non-stability of the balanced property
for solutions of the Strominger system with $\alpha'>0$,
we need to find an invariant complex structure $J$ on $M$
admitting such solutions and for which the dimension of the Bott-Chern cohomology group $H^{2,2}_{\rm BC}(M,J)$
does not attain a minimum in the space of all invariant complex structures on $M$.
Indeed, such a choice is necessary for performing later an appropriate holomorphic deformation, since
by \cite[Proposition 4.1]{AU} if $\{X_t\}_{t\in (-\varepsilon,\varepsilon)}$, $\varepsilon > 0$, is any
differentiable family of deformations of a compact balanced manifold $X_0$ of complex dimension $n$
such that the upper-semi-continuous function $t \mapsto  \dim_\mathbb{C} H^{n-1,n-1}_{\rm BC}(X_t)$
is constant, then $X_t$ also admits a balanced metric for any $t$ close enough to $0$.

From \cite[Appendix]{LUV} and \cite[Table 2]{AFR}, it follows that the minimum value of the dimension of the Bott-Chern
cohomology group $H^{2,2}_{\rm BC}(M,J)$ in the space of invariant complex
structures $J$ on the manifold $M$ is equal to 6.
By \cite[Theorem 2.6]{R1}, the complex structure of any sufficiently small deformation
of $(M,J)$ is again invariant.
So, if $(X_\mathbf{t})_{\mathbf{t} \in \Delta}$
is any (sufficiently small) holomorphic deformation of a complex manifold $X_0=(M,J)$
such that $\dim_\mathbb{C} H^{2,2}_{\rm BC}(M,J)=6$,
then the function $\mathbf{t} \mapsto  \dim_\mathbb{C} H^{2,2}_{\rm BC}(X_\mathbf{t})$
is constant.

In conclusion, we need to find a complex structure admitting balanced metrics for which the dimension of its
Bott-Chern cohomology group of bidegree (2,2) is at least 7.
But such complex structures are already classified in \cite[Table~1]{LUV}.
In fact, up to isomorphism of the complex structure (and apart from the Iwasawa manifold)
we are led to consider a complex structure $J^s$ in the following family:
\begin{equation}\label{ecus-JD}
d\omega^1=d\omega^2=0,\quad d\omega^3=\omega^{12} +\omega^{1\bar{1}} -s^2\, \omega^{2\bar{2}}, \quad\quad  s \in (0,1/2).
\end{equation}

From now on, given any $s \in (0,1/2)$, we will consider the compact complex manifold $X^s=(M,J^s)$ of complex dimension 3, where
$M$ is the (real) nilmanifold underlying the Iwasawa manifold and $J^s$ is the (integrable almost) complex structure defined by
the (1,0)-basis $\{\omega^1,\omega^2,\omega^3\}$ satisfying the equations $\eqref{ecus-JD}$.

\smallskip

In the following result we prove that any compact complex manifold $X^s=(M,J^s)$ has the Strominger property
with positive slope parameter in the anomaly cancellation equation.

\begin{prop}\label{X0-Strominger}
For every $s\in (0,\frac{1}{2})$, the compact complex manifold $X^s=(M,J^s)$ admits a solution to the Strominger system
with $\alpha'>0$.
\end{prop}

\begin{proof}
Firstly, notice that it follows directly from \eqref{ecus-JD} that the $(3,0)$-form $\Psi_s=s\,\omega^{123}$ is holomorphic,
so $X^s$ has holomorphically trivial canonical bundle.

Let us consider the $J^s$-Hermitian metric $F_s$ on $X^s$ given by
\begin{equation}\label{metrica-Js}
F_s= \frac{i}{2} (\omega^{1\bar{1}} +s^2\, \omega^{2\bar{2}}+\omega^{3\bar{3}}).
\end{equation}
Using the structure equations \eqref{ecus-JD}, we have
\begin{eqnarray*}
-4 \, d F_s \wedge F_s \!\!\!&=&\!\!\! d \omega^{3\bar{3}} \wedge (\omega^{1\bar{1}} +s^2\, \omega^{2\bar{2}}+\omega^{3\bar{3}})\\
  \!\!\!&=&\!\!\!(\omega^{12\bar{3}}+\omega^{1\bar{1}\bar{3}} -s^2\, \omega^{2\bar{2}\bar{3}} + \omega^{1\bar{1}3} -s^2\, \omega^{2\bar{2}3}
  -\omega^{3\bar{1}\bar{2}})\wedge (\omega^{1\bar{1}} +s^2\, \omega^{2\bar{2}})\\
  \!\!\!&=&\!\!\!(\omega^{1\bar{1}\bar{3}} + \omega^{1\bar{1}3})\wedge (s^2\, \omega^{2\bar{2}})
  -s^2\, (\omega^{2\bar{2}\bar{3}} + \omega^{2\bar{2}3})\wedge \omega^{1\bar{1}}\\
  \!\!\!&=&\!\!\!0.
\end{eqnarray*}
Hence, $d F_s^2 = 0$, i.e. the metric $F_s$ is a balanced metric on $X_s$.

The torsion 3-form $T_s=J^s dF_s$ is given by
$$
T_s= -\frac{1}{2} (\omega^{12\bar{3}}-\omega^{1\bar{1}\bar{3}} +s^2\, \omega^{2\bar{2}\bar{3}} + \omega^{1\bar{1}3} -s^2\, \omega^{2\bar{2}3}+\omega^{3\bar{1}\bar{2}}).
$$
By \eqref{ecus-JD} we get
\begin{equation}\label{dif-torsion-Ts}
dT_s=d J^s dF_s=\frac{2\,s^2+1}{s^4}\, \omega^{1\bar{1}2\bar{2}}.
\end{equation}

We will consider the Strominger-Bismut connection in the anomaly cancellation condition.
The curvature forms can be obtained more easily in a basis adapted to the SU(3)-structure $(F_s,\Psi_s)$.
Let $\{e^k\}_{k=1}^6$ be the basis of real 1-forms
defined as
$$
e^1= \Real \omega^1, \ \   e^{2} = \Imag \omega^1, \ \  e^3= \Real (s\,\omega^2), \ \
e^{4} = \Imag (s\,\omega^2), \ \  e^5= \Real \omega^3, \ \   e^{6} = \Imag \omega^3,
$$
that is, $\omega^1=e^1+i\, e^2$, $s\, \omega^2=e^3+i\, e^4$, and $\omega^3=e^5+i\, e^6$.
Hence, the complex structure $J^s$ and the SU(3)-structure $(F_s,\Psi_s)$ express as
$$
\begin{array}{lcl}
\zzz & &\zzz J^s e^1=-e^2,\ J^s e^3=-e^4,\ J^s e^5=-e^6,\\[4pt]
\zzz & &\zzz F_s=e^{12}+e^{34}+e^{56}, \\[4pt]
\zzz & &\zzz \Psi_s=(e^{1}+i\,e^{2})\wedge(e^{3}+i\,e^{4})\wedge(e^{5}+i\,e^{6}).
\end{array}
$$

In this adapted basis the curvature 2-forms of the Strominger-Bismut connection, which for simplicity
we will denote by $\Omega^i_j$ instead of $(\Omega_s)^i_j$, are
\begin{eqnarray*}
\Omega^1_2 \!\!\!&=&\!\!\! -4(e^{12}-e^{34}) +\frac{2}{s} (e^{14} + e^{23}) +\frac{2}{s^2} \,e^{34},\\
\Omega^1_3 \!\!\!&=&\!\!\! \Omega^2_4 = -\frac{1}{s^2} (e^{13} + e^{24}),\\
\Omega^1_4 \!\!\!&=&\!\!\! -\Omega^2_3 = -\frac{1}{s^2} (e^{14} - e^{23}),\\
\Omega^1_5 \!\!\!&=&\!\!\! \Omega^2_6 = - \frac{2}{s} e^{46},\\
\Omega^1_6 \!\!\!&=&\!\!\! -\Omega^2_5 = \frac{2}{s} e^{36},\\
\Omega^3_4 \!\!\!&=&\!\!\! 4 (e^{12}-e^{34}) -\frac{2}{s} (e^{14} + e^{23}) +\frac{2}{s^2} \,e^{12},\\
\Omega^3_5 \!\!\!&=&\!\!\! \Omega^4_6 = - \frac{2}{s} e^{26},\\
\Omega^3_6 \!\!\!&=&\!\!\! -\Omega^4_5 = \frac{2}{s} e^{16},\\
\Omega^5_6 \!\!\!&=&\!\!\! -\Omega^1_2-\Omega^3_4 = -\frac{2}{s^2} (e^{12}+e^{34}).
\end{eqnarray*}

Now, a direct calculation shows that
$${\rm tr}\,\Omega_s\wedge\Omega_s=\sum_{1\leq i<j\leq 6} \Omega^i_j\wedge\Omega^i_j = -16\, \frac{4\,s^2+1}{s^6}\, e^{1234}
= 4\, \frac{4\,s^2+1}{s^4}\, \omega^{1\bar{1}2\bar{2}}.$$

Taking $A=0$ (see Remark~\ref{remark-Strominger-no-open} for solutions with non-flat instanton),
from the previous equality and \eqref{dif-torsion-Ts},
we get
$$\frac{2\,s^2+1}{s^4}\, \omega^{1\bar{1}2\bar{2}}
=dT_s=\frac{\alpha'}{4} ({\rm tr}\ \Omega_s\wedge\Omega_s - {\rm tr}\ \Omega^A\wedge\Omega^A)
=4\,\alpha'\, \frac{4\,s^2+1}{s^4} \, \omega^{1\bar{1}2\bar{2}}.
$$
Therefore, the anomaly cancellation equation is solved for positive $\alpha'=\frac{2\,s^2+1}{4(4\,s^2+1)}$.
\end{proof}

Next we want to perform an appropriate holomorphic deformation of the compact complex manifold $X^s=(M,J^s)$.
As we recalled above, since $J^s$ is an invariant complex structure,
by \cite[Theorem 2.6]{R1} any sufficiently small deformation of $J^s$ is given again by an invariant complex structure
on the nilmanifold $M$.
Notice also that the Dolbeault cohomology of $X^s$ is isomorphic to the Lie algebra cohomology, in particular
$$
H^{0,1}(X^s) \cong H^{0,1}(\mathfrak{h}_5, J^s) = \langle [ \omega^{\bar{1}}],[ \omega^{\bar{2}}] \rangle.
$$
Following an idea in \cite[Example 8]{MPPS}, we first consider invariant almost complex structures $J^s_\Phi$, which are
``near'' to $J^s$, defined by the following basis of (1,0)-forms
\begin{equation}\label{def-space-JD}
\left\{
\begin{array}{lcl}
\omega_\Phi^1 \zzz & = &\zzz \omega^1 + \Phi^1_1\, \omega^{\bar{1}} + \Phi^1_2\, \omega^{\bar{2}},\\[4pt]
\omega_\Phi^2 \zzz & = &\zzz \omega^2 + \Phi^2_1\, \omega^{\bar{1}} + \Phi^2_2\, \omega^{\bar{2}},\\[4pt]
\omega_\Phi^3 \zzz & = &\zzz \omega^3,
\end{array}
\right.
\end{equation}
where the coefficients $\Phi^1_1,\Phi^1_2, \Phi^2_1, \Phi^2_2\in \mathbb{C}$.
The integrability of the almost complex structure $J_\Phi^s$ is equivalent to the closedness of the (3,0)-form
$$
\omega_\Phi^{123}=\omega^{123}
+(\Phi^2_1\, \omega^{1\bar{1}} + \Phi^2_2\, \omega^{1\bar{2}} -\Phi^1_1\, \omega^{2\bar{1}} - \Phi^1_2\, \omega^{2\bar{2}})\wedge\omega^3
+ (\Phi^1_1 \Phi^2_2-\Phi^1_2 \Phi^2_1)\omega^{\bar{1}\bar{2}3}.
$$
A direct calculation using \eqref{ecus-JD} and \eqref{def-space-JD} shows that $J_\Phi^s$ is integrable
if and only if
\begin{equation}\label{integrability-JD}
\Phi^1_1 \Phi^2_2-\Phi^1_2 \Phi^2_1 + \Phi^1_2 +s^2\, \Phi^2_1=0.
\end{equation}

In the following result we prove the existence of a direction in the deformation space such that the resulting
compact complex manifolds do not admit any solution to the Strominger system.

\begin{thrm}\label{Strominger-no-open}
There exists a holomorphic family of compact complex manifolds $(X_\mathbf{t})_{\mathbf{t}\in \Delta}$ of complex dimension $3$,
where $\Delta=\{ \mathbf{t}\in \mathbb{C}\mid |\mathbf{t}|< \varepsilon \}$,
$\varepsilon>0$, such that $X_\mathbf{t}$ has holomorphically trivial canonical bundle for any $\mathbf{t}\in\Delta$,
and
\begin{enumerate}
\item[\rm (i)] $X_0$ admits a solution to the Strominger system with $\alpha'>0$, but
\item[\rm (ii)] $X_\mathbf{t}$ does not admit balanced metrics for any $\mathbf{t}\in \Delta\!-\!\{\Imag \mathbf{t}=0\}$.
\end{enumerate}
In particular, the Strominger property is not stable under small deformations of the complex structure.
\end{thrm}

\begin{proof}
Let $s \in (0,1/2)$ and consider $X=X^s$, where $X^s$ is the compact complex manifold constructed above. By Proposition~\ref{X0-Strominger}
we know that $X$ has the Strominger property with $\alpha'>0$. We will construct a small deformation $X_\mathbf{t}$ of $X_0=X$ for which
there are non-zero $\mathbf{t}$'s arbitrary close to $0$ such that
the compact complex manifold $X_\mathbf{t}$ does not admit any balanced metric.

To construct such a deformation, we first choose $\Phi^1_1=\Phi^2_2=0$ in \eqref{def-space-JD}.
Now, the integrability condition \eqref{integrability-JD}
is equivalent to $\Phi^2_1(\Phi^1_2 -s^2) - \Phi^1_2 =0$, where $\Phi^1_2, \Phi^2_1\in \mathbb{C}$.
We consider the open disc $\Delta=\Delta(0,s^2)$ around $0 \in \mathbb{C}$ and of radius $s^2$,
and we take $\Phi^1_2 \in \Delta$.

From now on, we denote $\Phi^1_2=\mathbf{t}$. Since $\mathbf{t}\in\Delta(0,s^2)\subset \mathbb{C}$,
we have $\mathbf{t}-s^2\not=0$ and the integrability condition \eqref{integrability-JD}
reduces to $\Phi^2_1=\mathbf{t}/(\mathbf{t}-s^2)$.
So, the system \eqref{def-space-JD} is written as follows:
\begin{equation}\label{def-space-JD-choose}
\omega_{\mathbf{t}}^1 = \omega^1 + \mathbf{t}\, \omega^{\bar{2}},\quad\
\omega_{\mathbf{t}}^2 = \omega^2 + \frac{\mathbf{t}}{\mathbf{t}-s^2}\, \omega^{\bar{1}},\quad\
\omega_{\mathbf{t}}^3 = \omega^3, \quad\quad \mbox{ for } \mathbf{t}\in\Delta(0,s^2).
\end{equation}
Hence, we have, for each $\mathbf{t}\in\Delta(0,s^2)$,
an invariant complex structure $J_{\mathbf{t}}$ defined
by this basis of bidegree (1,0) on the real nilmanifold $M$.
We denote by $X_{\mathbf{t}}$ the compact complex manifold~$(M,J_{\mathbf{t}})$.

Notice that since the Dolbeault cohomology group $H^{0,1}(X_0) = \langle [ \omega^{\bar{1}}],[ \omega^{\bar{2}}] \rangle$,
one can consider $X_{\mathbf{t}}$ as a small deformation of the compact complex manifold $X_0$ given by
$$
\frac{\mathbf{t}}{\mathbf{t}-s^2}\,\frac{\partial}{\partial z_2} \otimes  \omega^{\bar{1}}
+ \mathbf{t}\,\frac{\partial}{\partial z_1} \otimes \omega^{\bar{2}} \in H^{0,\,1}(X_0,\,T^{1,0}X_0).
$$

Now, using \eqref{ecus-JD}, the complex structure equations of $X_{\mathbf{t}}$ are
$$
d\omega_{\mathbf{t}}^1=d\omega_{\mathbf{t}}^2=0,\quad\
d\omega_{\mathbf{t}}^3 =
a(\mathbf{t})\, \omega_{\mathbf{t}}^{12}+b(\mathbf{t})\, \omega_{\mathbf{t}}^{1\bar{1}}+c(\mathbf{t})\, \omega_{\mathbf{t}}^{2\bar{2}} ,
$$
where
$$
a(\mathbf{t})=(\mathbf{t}-s^2) \frac{\bar{\mathbf{t}}(1-\bar{\mathbf{t}})-s^2}{|\mathbf{t}-|\mathbf{t}|^2-s^2|^2}, \quad
b(\mathbf{t})=-s^2 \frac{\bar{\mathbf{t}}(1-\mathbf{t})-s^2}{|\mathbf{t}-|\mathbf{t}|^2-s^2|^2}, \quad
c(\mathbf{t})=|\mathbf{t}-s^2|^2 \frac{\mathbf{t}(1-\bar{\mathbf{t}})-s^2}{|\mathbf{t}-|\mathbf{t}|^2-s^2|^2}.
$$
Notice that $a(0)=b(0)=1$ and $c(0)=-s^2$, so the coefficients are non-zero for
sufficiently small $\mathbf{t}$. Furthermore, a direct calculation shows that $\mathbf{t} \not= |\mathbf{t}|^2+s^2$
and $a(\mathbf{t})$, $b(\mathbf{t})$, $c(\mathbf{t})$ do not vanish for any $\mathbf{t} \in \Delta(0,s^2)$.

Let us consider the new basis $\{\eta_{\mathbf{t}}^1, \eta_{\mathbf{t}}^2, \eta_{\mathbf{t}}^3\}$
of bidegree (1,0) for $X_{\mathbf{t}}$ defined by
$$
\eta_{\mathbf{t}}^1=a(\mathbf{t})\,\omega_{\mathbf{t}}^1, \quad
\eta_{\mathbf{t}}^2=\frac{|a(\mathbf{t})|^2}{b(\mathbf{t})}\,\omega_{\mathbf{t}}^2, \quad
\eta_{\mathbf{t}}^3=\frac{|a(\mathbf{t})|^2}{b(\mathbf{t})}\,\omega_{\mathbf{t}}^3.
$$
This basis satisfies
$$
d\eta_{\mathbf{t}}^1=d\eta_{\mathbf{t}}^2=0, \quad\
d\eta_{\mathbf{t}}^3 = \eta_{\mathbf{t}}^{12} + \eta_{\mathbf{t}}^{1\bar{1}}
+D(\mathbf{t})\, \eta_{\mathbf{t}}^{2\bar{2}} ,
$$
where
$$
D(\mathbf{t})=\frac{\overline{b(\mathbf{t})}\, c(\mathbf{t})}{|a(\mathbf{t})|^2}=
-s^2\, \frac{(\mathbf{t}-|\mathbf{t}|^2-s^2)^2}{|\mathbf{t}-\mathbf{t}^2-s^2|^2}.
$$

Since every $X_\mathbf{t}$ is a nilmanifold endowed with an invariant complex structure,
by symmetrization \cite{FG}, the compact complex manifold $X_\mathbf{t}$
has a balanced metric if and only if it has an invariant one.
Now, by \cite[Proposition 2.3]{UV2}
any invariant Hermitian metric $F$ on $X_\mathbf{t}$ can be written as
$$
2F=i(\eta_\mathbf{t}^{1\bar1}+p^2\,\eta_\mathbf{t}^{2\bar2}+q^2\,\eta_\mathbf{t}^{3\bar3})
+ u\,\eta_\mathbf{t}^{1\bar2} - \bar u\,\eta_\mathbf{t}^{2\bar1},
$$
for some $p,q \in \mathbb{R}^*$, $u \in \mathbb{C}$ with $p^2>|u|^2$.
Moreover, the metric $F$ is balanced if and only if
$$
p^2+ D(\mathbf{t}) =0.
$$
Next we prove that the previous condition does not hold for any $\mathbf{t}\in \Delta(0,s^2)\!-\!\{\Imag \mathbf{t}=0\}$.
Indeed, a direct calculation shows that $D(\mathbf{t})$ is real if and only if
\begin{equation}\label{condis}
t_2\, (t_1^2+t_2^2-t_1+s^2) =0,
\end{equation}
where $t_1=\Real(\mathbf{t})$ and $t_2=\Imag(\mathbf{t})$, i.e. $\mathbf{t}=t_1+i\, t_2$.
It is easy to see that $t_1^2+t_2^2-t_1+s^2 \not=0$ for any $\mathbf{t}\in \Delta(0,s^2)$, so
the condition \eqref{condis} is satisfied if and only if $t_2=0$.
Therefore, for any $\mathbf{t}\in \Delta(0,s^2)\!-\!\{t_2=0\}$ we have that $\Imag D(\mathbf{t})\not=0$
and the balanced condition is not satisfied.
\end{proof}

\section{The Strominger property with positive $\alpha'$ is not closed}\label{sec-closed}

\noindent
Our aim in this section is to prove that the Strominger property with positive slope parameter is not closed
under holomorphic deformations.

We will focus on the complex geometry of the nilmanifold corresponding to the nilpotent Lie algebra
$\frh_4=(0,0,0,0,12,14+23)$ in the notation of \cite{Sal}.
This 6-dimensional compact manifold, which we will denote by $N$, has many
invariant complex structures $J$.
By \cite{Sal} the canonical bundle of any $(N,J)$ is holomorphically trivial.

Let $J_0$ be the invariant complex structure on $N$ defined by a (1,0)-basis $\{\eta^1,\eta^2,\eta^3\}$ satisfying
the equations
\begin{equation}\label{ecus}
d\eta^1=d\eta^2=0,\quad d\eta^3=\frac{i}{2} \eta^{1\bar{1}} +\frac12 \eta^{1\bar{2}} +\frac12 \eta^{2\bar{1}}.
\end{equation}
Maclaughlin, Pedersen, Poon and Salamon studied in \cite[Example 8]{MPPS} the deformation parameter space of $J_0$.
Any invariant complex structure sufficiently near to $J_0$ has a basis of (1,0)-forms $\{\eta_\Phi^1,\eta_\Phi^2,\eta_\Phi^3\}$
that can be written as
\begin{equation}\label{def-space}
\left\{
\begin{array}{lcl}
\eta_\Phi^1 \zzz & = &\zzz \eta^1 + \Phi^1_1\, \eta^{\bar{1}} + \Phi^1_2\, \eta^{\bar{2}},\\[4pt]
\eta_\Phi^2 \zzz & = &\zzz \eta^2 +\Phi^2_1\, \eta^{\bar{1}} + \Phi^2_2\, \eta^{\bar{2}},\\[4pt]
\eta_\Phi^3 \zzz & = &\zzz \eta^3 + \Phi^3_3\, \eta^{\bar{3}},
\end{array}
\right.
\end{equation}
where the coefficients $\Phi^1_1,\Phi^1_2, \Phi^2_1, \Phi^2_2, \Phi^3_3\in \mathbb{C}$ are sufficiently small,
and they satisfy the integrability condition.
Notice that the latter is equivalent to the closedness of the (3,0)-form
$\eta_\Phi^{123}$, and using \eqref{ecus} it is given by
\begin{equation}\label{integrability}
i(1+\Phi^3_3)\Phi^1_2=(1-\Phi^3_3)(\Phi^1_1-\Phi^2_2).
\end{equation}

We will consider a holomorphic family constructed in \cite{COUV} to answer a question in \cite{Pop2}
about the closedness of the balanced property.
Let us take $\Phi^1_1=\mathbf{t}$, $\Phi^1_2=-i\mathbf{t}$ and $\Phi^2_1=\Phi^2_2=\Phi^3_3=0$ in
the parameter space~\eqref{def-space}, where $\mathbf{t}\in \Delta=\Delta(0,1)=\{ \mathbf{t}\in \mathbb{C}\mid |\mathbf{t}|<1 \}$.
That is to say, for each $\mathbf{t}\in \Delta$, we consider the basis $\{\eta_\mathbf{t}^1,\eta_\mathbf{t}^2,\eta_\mathbf{t}^3\}$
of (1,0)-forms given by
\begin{equation}\label{def-space-particular}
\eta_\mathbf{t}^1=\eta^1 +\mathbf{t} \eta^{\bar{1}} -i \mathbf{t} \eta^{\bar{2}},\quad \eta_\mathbf{t}^2=\eta^2,\quad \eta_\mathbf{t}^3=\eta^3.
\end{equation}
This basis defines an invariant complex structure $J_\mathbf{t}$ on the nilmanifold $N$,
and we denote by $X_\mathbf{t}$ the corresponding compact complex manifold $(N,J_\mathbf{t})$.

In \cite[Theorem 5.9]{COUV} it is proved that a compact complex manifold in the holomorphic family $(X_\mathbf{t})_{\mathbf{t}\in \Delta}$
has a balanced metric if and only if $\mathbf{t} \in \Delta-\{0\}$.
Thus, there does not exist any balanced metric on the central limit $X_0=(N,J_0)$, and the
balanced property is not closed by holomorphic deformations.

In the following result we prove that any compact complex manifold $X_\mathbf{t}$, $\mathbf{t} \not=0$, admits in addition a
solution to the Strominger system with positive slope parameter $\alpha'$.

\begin{thrm}\label{no-closed}
There exists a holomorphic family of compact complex manifolds $(X_\mathbf{t})_{\mathbf{t}\in \Delta}$,
where $\Delta=\{ \mathbf{t}\in \mathbb{C}\mid |\mathbf{t}|< 1 \}$, such that $X_\mathbf{t}$ has holomorphically trivial canonical bundle for any $\mathbf{t}\in\Delta$,
and
\begin{enumerate}
\item[\rm (i)] $X_\mathbf{t}$ admits a solution to the Strominger system with $\alpha'>0$ for every $\mathbf{t}\in \Delta\!-\!\{0\}$, but
\item[\rm (ii)] $X_0$ does not admit any balanced Hermitian metric.
\end{enumerate}
In particular, the Strominger property with positive slope parameter is not closed under holomorphic deformations.
\end{thrm}

\begin{proof}
A direct calculation using \eqref{ecus} shows that the complex structure equations for $X_\mathbf{t}$,
with respect to the basis \eqref{def-space-particular}, are
\begin{equation}\label{ecus-disc-a}
d\eta_\mathbf{t}^1=d\eta_\mathbf{t}^2=0,\quad
2(1-|\mathbf{t}|^2)\, d\eta_\mathbf{t}^3=2\bar{\mathbf{t}} \eta_\mathbf{t}^{12}+i \eta_\mathbf{t}^{1\bar{1}} +\eta_\mathbf{t}^{1\bar{2}}
+ \eta_\mathbf{t}^{2\bar{1}} -i|\mathbf{t}|^2 \eta_\mathbf{t}^{2\bar{2}},
\end{equation}
for each $\mathbf{t}\in\Delta$.

By \cite[Theorem 5.9]{COUV}, the compact complex manifold $X_0$ does not admit any balanced metric.
For each $\mathbf{t}\in \Delta- \{0\}$, we consider the Hermitian metric $F_{\mathbf{t},r}$ on $X_\mathbf{t}$ given by
\begin{equation}\label{metric-disc-a}
F_{\mathbf{t},r} = \frac{i}{2}\, ( \eta_\mathbf{t}^{1\bar{1}} + |\mathbf{t}|^2\, \eta_\mathbf{t}^{2\bar{2}} + r^2\, \eta_\mathbf{t}^{3\bar{3}}),
\end{equation}
where $r \in \mathbb{R}^*$. Now,
\begin{eqnarray*}
-4 \, d F_{\mathbf{t},r} \wedge F_{\mathbf{t},r} \!\!\!&=&\!\!\!
r^2\, d \eta_\mathbf{t}^{3\bar{3}} \wedge (\eta_\mathbf{t}^{1\bar{1}} + |\mathbf{t}|^2\, \eta_\mathbf{t}^{2\bar{2}} + r^2\, \eta_\mathbf{t}^{3\bar{3}})\\
  \!\!\!&=&\!\!\!r^2\, (d \eta_\mathbf{t}^{3}\wedge \eta_\mathbf{t}^{\bar{3}} - \eta_\mathbf{t}^3\wedge d \eta_\mathbf{t}^{\bar{3}}) \wedge (\eta_\mathbf{t}^{1\bar{1}} + |\mathbf{t}|^2\, \eta_\mathbf{t}^{2\bar{2}})\\
  \!\!\!&=&\!\!\!0,
\end{eqnarray*}
because $d \eta_\mathbf{t}^{3}\wedge (\eta_\mathbf{t}^{1\bar{1}} + |\mathbf{t}|^2\, \eta_\mathbf{t}^{2\bar{2}})=0$ by \eqref{ecus-disc-a}.
Hence, the metric $F_{\mathbf{t},r}$ is a balanced metric on $X_{\mathbf{t}}$ for any
$0<|\mathbf{t}|<1$ and $r\in \mathbb{R}^*$.

For each $\mathbf{t}\in \Delta- \{0\}$, we consider the real basis of 1-forms
$\{e^1,\ldots,e^6\}$ defined as
$$
e^1= \Real \eta_\mathbf{t}^1, \ \   e^{2} = \Imag \eta_\mathbf{t}^1, \ \  e^3= \Real (|\mathbf{t}|\,\eta_\mathbf{t}^2), \ \
e^{4} = \Imag (|\mathbf{t}|\,\eta_\mathbf{t}^2), \ \  e^5= \Real (r\, \eta_\mathbf{t}^3), \ \   e^{6} = \Imag (r\, \eta_\mathbf{t}^3),
$$
that is,
$$
e^1+i\, e^2 = \eta_\mathbf{t}^1,\quad
e^3+i\, e^4 = |\mathbf{t}|\,\eta_\mathbf{t}^2,\quad
e^5+i\,e^6 = r\, \eta_\mathbf{t}^3.
$$
Hence, the complex structure $J_\mathbf{t}$ and the SU(3)-structure
$(F_{\mathbf{t},r},\Psi_{\mathbf{t},r}= r\,|\mathbf{t}|\,\eta_\mathbf{t}^{123})$ express
in this basis as
\begin{equation}\label{adapted-basis-bis}
\begin{array}{lcl}
\zzz & &\zzz J_\mathbf{t} e^1=-e^2,\ J_\mathbf{t} e^3=-e^4,\ J_\mathbf{t} e^5=-e^6,\\[4pt]
\zzz & &\zzz F_{\mathbf{t},r}=e^{12}+e^{34}+e^{56}, \\[4pt]
\zzz & &\zzz \Psi_{\mathbf{t},r}=(e^{1}+i\,e^{2})\wedge(e^{3}+i\,e^{4})\wedge(e^{5}+i\,e^{6}).
\end{array}
\end{equation}

Moreover, it follows from \eqref{ecus-disc-a} that the structure equations
in the adapted basis $\{e^k\}_{k=1}^6$ are
\begin{equation}\label{ecus-reales-disc-a}
\left\{
  \begin{aligned}
  &d e^1= d e^2= d e^3= d e^4= 0, \\
  &d e^5= \frac{r}{1-|\mathbf{t}|^2}(e^{12}-e^{34})
  +\frac{r\, t_1}{|\mathbf{t}|(1-|\mathbf{t}|^2)}(e^{13}-e^{24})
  +\frac{r\, t_2}{|\mathbf{t}|(1-|\mathbf{t}|^2)}(e^{14}+e^{23}) ,\\
  &d e^6= -\frac{r\, t_2}{|\mathbf{t}|(1-|\mathbf{t}|^2)}(e^{13}-e^{24})
  -\frac{r\,(1-t_1)}{|\mathbf{t}|(1-|\mathbf{t}|^2)} e^{14}
  +\frac{r\,(1+t_1)}{|\mathbf{t}|(1-|\mathbf{t}|^2)}e^{23},
  \end{aligned}
\right.
\end{equation}
where $t_1=\Real(\mathbf{t})$ and $t_2=\Imag(\mathbf{t})$, i.e. $\mathbf{t}=t_1+i\, t_2$.

Using \eqref{ecus-reales-disc-a}, we get
that the torsion 3-form $T_{\mathbf{t},r}=J_\mathbf{t} dF_{\mathbf{t},r}$ is given by
$$
\begin{aligned}
  & \frac{|\mathbf{t}|(1-|\mathbf{t}|^2)}{r}\, T_{\mathbf{t},r}= |\mathbf{t}|\, e^{125}-t_1\, e^{135}+t_2\, e^{136}-t_2\, e^{145}
-(1+t_1)\, e^{146} \\[5pt]
&\phantom{\frac{|\mathbf{t}|(1-|\mathbf{t}|^2)}{r}\, T_{\mathbf{t},r}=}
-t_2\, e^{235}+(1-t_1) e^{236}+t_1\, e^{245}-t_2\, e^{246}- |\mathbf{t}| e^{345} .
\end{aligned}
$$

Again from \eqref{ecus-reales-disc-a} we get
$$
dT_{\mathbf{t},r}= -\frac{2\, r^2(1+3|\mathbf{t}|^2)}{|\mathbf{t}|^2(1-|\mathbf{t}|^2)^2}\, e^{1234}
= \frac{ r^2(1+3|\mathbf{t}|^2)}{2 (1-|\mathbf{t}|^2)^2}\, \eta_\mathbf{t}^{1\bar{1}2\bar{2}}.
$$
In the last equality we have used the relation $e^{1234}=-\frac{|\mathbf{t}|^2}{4} \, \eta_\mathbf{t}^{1\bar{1}2\bar{2}}$.

A long but direct calculation shows that, in the adapted basis $\{e^k\}_{k=1}^6$, the curvature 2-forms of the Strominger-Bismut connection,
which for simplicity
we will denote by $\Omega^i_j$ instead of $(\Omega_{\mathbf{t},r})^i_j$, are:
\begin{eqnarray*}
\rho({\mathbf{t},r})\, \Omega^1_2 \!\!\!&=&\!\!\!
-|\mathbf{t}|^2\, e^{12} - t_1 |\mathbf{t}|\, (e^{13} - e^{24}) - t_2 |\mathbf{t}|\, (e^{14} + e^{23})  + 3 |\mathbf{t}|^2\, e^{34},\\[3pt]
\rho({\mathbf{t},r})\,  \Omega^1_3 \!\!\!&=&\!\!\!   - |\mathbf{t}|^2\, (e^{13} + e^{24}) +2 |\mathbf{t}|\, e^{56},\\[3pt]
\rho({\mathbf{t},r})\, \Omega^1_4 \!\!\!&=&\!\!\!
- t_2\, (e^{13} - e^{24}) -(1-t_1+|\mathbf{t}|^2) e^{14} +(1+t_1+|\mathbf{t}|^2) e^{23},\\[3pt]
\rho({\mathbf{t},r})\, \Omega^1_5 \!\!\!&=&\!\!\!
- t_2(e^{16}+|\mathbf{t}|\, e^{35}) +t_1(e^{26}+|\mathbf{t}|\, e^{45}),\\[3pt]
\rho({\mathbf{t},r})\, \Omega^1_6 \!\!\!&=&\!\!\!
- t_1(e^{16}+|\mathbf{t}|\, e^{35}) -t_2(e^{26}+|\mathbf{t}|\, e^{45}),\\[3pt]
\rho({\mathbf{t},r})\, \Omega^3_4 \!\!\!&=&\!\!\!
3|\mathbf{t}|^2\, e^{12} + t_1 |\mathbf{t}|\, (e^{13} - e^{24}) + t_2 |\mathbf{t}|\, (e^{14} + e^{23})  - |\mathbf{t}|^2\, e^{34},\\[3pt]
\rho({\mathbf{t},r})\, \Omega^3_5 \!\!\!&=&\!\!\!
- t_2(|\mathbf{t}|\, e^{15}-e^{36}) +t_1(|\mathbf{t}|\, e^{25}-e^{46}),\\[3pt]
\rho({\mathbf{t},r})\, \Omega^3_6 \!\!\!&=&\!\!\!
- t_1(|\mathbf{t}|\, e^{15}-e^{36}) -t_2(|\mathbf{t}|\, e^{25}-e^{46}),
\end{eqnarray*}
where $\rho({\mathbf{t},r})=\frac{|\mathbf{t}|^2 (1-|\mathbf{t}|^2)^2}{r^2}$.
The other curvature 2-forms are given by the relations:
$\Omega^2_3=-\Omega^1_4$, $\Omega^2_4=\Omega^1_3$, $\Omega^2_5=-\Omega^1_6$, $\Omega^2_6=\Omega^1_5$,
$\Omega^4_5=-\Omega^3_6$, $\Omega^4_6=\Omega^3_5$, and $\Omega^5_6 = -\Omega^1_2-\Omega^3_4$.

\smallskip

Now, we have
$${\rm tr}\,\Omega_{\mathbf{t},r}\wedge\Omega_{\mathbf{t},r}=\sum_{1\leq i<j\leq 6} \Omega^i_j\wedge\Omega^i_j =
-4 \, r^4\, \frac{\, 1+|\mathbf{t}|^2+2|\mathbf{t}|^4 \,}{|\mathbf{t}|^4 (1-|\mathbf{t}|^2)^4} \ e^{1234}
=
 r^4\, \frac{\, 1+|\mathbf{t}|^2+2|\mathbf{t}|^4 \,}{|\mathbf{t}|^2 (1-|\mathbf{t}|^2)^4} \, \eta_\mathbf{t}^{1\bar{1}2\bar{2}}.
$$

Taking the trivial instanton $A=0$ (see Remark~\ref{remark-Strominger-no-closed} for solutions with non-flat instanton),
we conclude that for each $\mathbf{t}$ such that $0<|\mathbf{t}|<1$, there exists a solution to the anomaly cancellation equation
with positive $\alpha'$.
Indeed, from the equality
$$
\frac{r^2(1+3|\mathbf{t}|^2)}{2(1-|\mathbf{t}|^2)^2}\, \eta_\mathbf{t}^{1\bar{1}2\bar{2}}
=dT_{\mathbf{t},r}
= \frac{\alpha'}{4} ({\rm tr}\ \Omega_{\mathbf{t},r}\wedge\Omega_{\mathbf{t},r} - {\rm tr}\ \Omega^A\wedge\Omega^A)
=\alpha'\, \frac{r^4(1+|\mathbf{t}|^2+2|\mathbf{t}|^4)}{4|\mathbf{t}|^2(1-|\mathbf{t}|^2)^4}\, \eta_\mathbf{t}^{1\bar{1}2\bar{2}},
$$
we are led to the following positive value for the slope parameter:
$$
\alpha'= \frac{2\, |\mathbf{t}|^2(1+3|\mathbf{t}|^2)(1-|\mathbf{t}|^2)^2}{r^2(1+|\mathbf{t}|^2+2|\mathbf{t}|^4)} >0.
$$
\end{proof}

\section{The negative slope parameter case}\label{sec-negative}

\noindent
In this section we focus on the Strominger property negative slope parameter $\alpha'$.

The first solutions of this kind were obtained
in \cite{CCDLMZ} on the Iwasawa manifold by using a certain abelian instanton $A$
(see the proof of Theorem~\ref{negative-no-open-no-closed} below).
In \cite{FIUV09} it is proved that there is no invariant solution
to the Strominger system with $\alpha'>0$ with respect to the Chern, Strominger-Bismut, Levi-Civita or the $\nabla^-=\nabla^g-\frac12H$ connections in the
anomaly cancellation equation.
However, to our knowledge it is an open question if the Iwasawa manifold can admit or not any solution to the Strominger system with
$\alpha'>0$.

Explicit solutions with $\alpha'<0$ on nilmanifolds, solvmanifolds and on compact quotient
of SL(2,$\mathbb{C}$) are constructed in \cite{Fei-Yau,FIUV09,FIUV2014,OUV1,UV2}, both with trivial or non-flat instanton,
and with respect to a family of connections in the anomaly cancellation equation.
On the other hand, in~\cite{Fu-Yau-alpha-negative} Fu and Yau obtained solutions to the
Strominger system with negative $\alpha'$ on a class of complex 3-dimensional manifolds
constructed by Goldstein and Prokushkin \cite{GP}.
More recently, Phong,  Picard and Zhang solved in \cite{PPZ}
the Fu-Yau equation with negative slope parameter in arbitrary dimensions
and studied its relation to a certain modification of the Strominger system.

An interesting question is what are the
differences and similarities of the geometric properties of the positive and the negative slope parameter cases.
Notice that a non-stability result similar to
Theorem~\ref{Strominger-no-open} in the case of negative $\alpha'$
follows directly from \cite{AB} and \cite{CCDLMZ}.
Indeed, by \cite{CCDLMZ} the Iwasawa manifold has a solution to the Strominger system with $\alpha'<0$, and by \cite{AB}
there are small holomorphic deformations of the Iwasawa manifold not admitting any balanced metric.
However, we do not know of any result about closedness of the Strominger property for negative $\alpha'$.
Next we prove a similar result to Theorem~\ref{no-closed}.

\begin{thrm}\label{negative-no-open-no-closed}
The Strominger property with negative slope parameter is neither open nor closed under holomorphic deformations.
\end{thrm}

\begin{proof}
As we noticed above, there are small deformations of the Iwasawa manifold showing that the Strominger property with $\alpha'<0$
is not stable. In the Remark~\ref{other-example-no-open-negative} below, we further study the compact complex manifolds $X^s$
constructed in Proposition~\ref{X0-Strominger}, and show that any $X^s$ also admits a solution to the Strominger system with $\alpha'<0$.
Hence, one has a family
of explicit examples of compact complex manifolds on which the property is not stable.

\smallskip

For the proof of non-closedness of the Strominger property with negative slope parameter, we will consider
the family $X_\mathbf{t}$ constructed in the proof of Theorem~\ref{no-closed}.
Next, we show that $X_\mathbf{t}$ admits a solution to the Strominger system with $\alpha'<0$ for every $0<|\mathbf{t}|<1$.

Since the compact complex manifold $X_\mathbf{t}$ has the structure of a torus bundle over a complex 2-torus, we can consider
the instanton constructed
by
Cardoso, Curio, Dall'Agata, Lust, Manousselis and Zoupanos in \cite{CCDLMZ}.
They considered an abelian field strength
configuration with (1,1)-form
$$
\mathcal{F} = if\, dz_1\wedge d\bar{z}_1 - if\, dz_2\wedge
d\bar{z}_2 + {\rm e}^{i\gamma} \sqrt{\frac14 - f^2}\, dz_1\wedge
d\bar{z}_2 - {\rm e}^{-i\gamma} \sqrt{\frac14 - f^2}\, dz_2\wedge
d\bar{z}_1,
$$
where the function $f$ satisfies
$$
i\partial_{z_2} f + \partial_{z_1}\!\!\left({\rm e}^{-i\gamma}
\sqrt{\frac14 - f^2} \right)=0, \quad\quad\  i\partial_{z_1} f +
\partial_{z_2}\!\!\left({\rm e}^{i\gamma} \sqrt{\frac14 - f^2}
\right)=0.
$$
Under these conditions one gets
\begin{equation}\label{Cardoso}
{\rm tr}\ \Omega^A\wedge\Omega^A= \mathcal{F}\wedge \mathcal{F} = -\frac12
dz_1\wedge dz_2\wedge d\bar{z}_1 \wedge d\bar{z}_2.
\end{equation}

Here $dz_1$ and $dz_2$ the (1,0)-forms at the level of the Lie group, which descend to the forms
$\eta_\mathbf{t}^1$ and $|\mathbf{t}|\eta_\mathbf{t}^2$ on the compact nilmanifold $X_\mathbf{t}=(N,J_\mathbf{t})$.
Notice that the Hermitian-Yang-Mills equation is satisfied, i.e.
$$
\Omega^A \wedge F_{\mathbf{t},r}^2=0, \quad \quad (\Omega^A)^{0,2}=(\Omega^A)^{2,0}=0.
$$
Thus, by \eqref{Cardoso} we have
${\rm tr}\, \Omega^A\wedge \Omega^A
= -\frac12\, dz_1\wedge dz_2\wedge d\bar{z}_1 \wedge d\bar{z}_2
= \frac{|\mathbf{t}|^2}{2}\, \eta_\mathbf{t}^{1\bar{1}2\bar{2}}$, and
the anomaly cancellation equation becomes
\begin{equation}\label{new-anomaly}
\begin{array}{lcl}
\displaystyle{\frac{r^2(1+3|\mathbf{t}|^2)}{2(1-|\mathbf{t}|^2)^2}\, \eta_\mathbf{t}^{1\bar{1}2\bar{2}}}
\zzz & = &\zzz \displaystyle{dT_{\mathbf{t},r}} \\[4pt]
\zzz & = &\zzz \displaystyle{\frac{\alpha'}{4} ({\rm tr}\ \Omega_{\mathbf{t},r}\wedge\Omega_{\mathbf{t},r} - {\rm tr}\ \Omega^A\wedge\Omega^A)} \\[8pt]
\zzz & = &\zzz \displaystyle{\frac{\alpha'}{4} \left( \frac{r^4(1+|\mathbf{t}|^2+2|\mathbf{t}|^4)}{|\mathbf{t}|^2(1-|\mathbf{t}|^2)^4} - \frac{|\mathbf{t}|^2}{2} \right)\, \eta_\mathbf{t}^{1\bar{1}2\bar{2}}.}
\end{array}
\end{equation}
This implies that the slope parameter is given by
$$
\alpha'=
\frac{4r^2|\mathbf{t}|^2(1+3|\mathbf{t}|^2)(1-|\mathbf{t}|^2)^2}{2\,r^4(1+|\mathbf{t}|^2+2|\mathbf{t}|^4)-|\mathbf{t}|^4(1-|\mathbf{t}|^2)^4}.
$$
Therefore, on $X_\mathbf{t}$, $0<|\mathbf{t}|<1$, we can take a balanced metric $F_{\mathbf{t},r}$ in~\eqref{metric-disc-a}
with $r\in \mathbb{R}^*$ small enough such that $r^4< \frac{|\mathbf{t}|^4(1-|\mathbf{t}|^2)^4}{2(1+|\mathbf{t}|^2+2|\mathbf{t}|^4)}$,
which ensures that $\alpha'<0$.
In conclusion, we have constructed a holomorphic family of compact complex manifolds $(X_\mathbf{t})_{\mathbf{t}\in \Delta}$,
where $\Delta=\{ \mathbf{t}\in \mathbb{C}\mid |\mathbf{t}|< 1 \}$, such that $X_\mathbf{t}$ has holomorphically trivial canonical bundle for any $\mathbf{t}\in\Delta$,
and
\begin{enumerate}
\item[\rm (i)] $X_\mathbf{t}$ admits a solution to the Strominger system with $\alpha'<0$ for every $\mathbf{t}\in \Delta\!-\!\{0\}$, but
\item[\rm (ii)] $X_0$ does not admit any balanced Hermitian metric.
\end{enumerate}
In particular, the Strominger property with negative slope parameter is not closed under holomorphic deformations.
\end{proof}

\begin{rmrk}\label{remark-Strominger-no-open}
{\rm
In the proof of Proposition~\ref{X0-Strominger} we have considered the trivial instanton $A$
in the anomaly cancellation equation. It is worth to remark that one can construct solutions with non-flat instanton
on every compact complex manifold $X^s=(M,J^s)$.
For that, we must first enlarge the space of balanced $J^s$-Hermitian metrics \eqref{metrica-Js} by
\begin{equation}\label{new-balanced-family}
F_{s,r}= \frac{i}{2} (\omega^{1\bar{1}} +s^2\, \omega^{2\bar{2}}+ r^2\,\omega^{3\bar{3}}),
\end{equation}
where $r\in \mathbb{R}^*$.
In fact, $F^2_{s,r}$ is a closed form, so $F_{s,r}$ defines a 1-parameter family of
balanced $J^s$-Hermitian metrics on $X^s$. A similar calculation as in the proof of Proposition~\ref{X0-Strominger}
shows that the torsion $T_{s,r}=J^s d F_{s,r}$ and the curvature $\Omega_{s,r}$ of the Strominger-Bismut connection satisfy
$$
dT_{s,r}= r^2\, \frac{2s^2+1}{s^4}\, \omega^{1\bar{1}2\bar{2}},\quad \quad
{\rm tr}\,\Omega_{s,r}\wedge\Omega_{s,r} = 4 r^4\, \frac{4s^2+1}{s^4}\, \omega^{1\bar{1}2\bar{2}}.
$$
Since the compact complex manifold $X^s$ has the structure of a torus bundle over a complex 2-torus, we can consider
the instanton constructed in \cite{CCDLMZ} and used in the proof of Theorem~\ref{negative-no-open-no-closed}.
In this case we consider $dz_1$ and $dz_2$ the (1,0)-forms at the level of the Lie group, which descend to the forms
$\omega^1$ and $s\omega^2$ on the compact nilmanifold $X^s=(M,J^s)$.
Thus, by \eqref{Cardoso} we have
${\rm tr}\, \Omega^A\wedge \Omega^A
= -\frac12\, dz_1\wedge dz_2\wedge d\bar{z}_1 \wedge d\bar{z}_2
= \frac{s^2}{2}\, \omega^{1\bar{1}2\bar{2}}$, and
the anomaly cancellation equation becomes
\begin{equation}\label{new-anomaly-1}
\begin{array}{lcl}
\displaystyle{r^2\,\frac{2s^2+1}{s^4}\, \omega^{1\bar{1}2\bar{2}}}
\zzz & = &\zzz \displaystyle{dT_{s,r}} \\[4pt]
\zzz & = &\zzz \displaystyle{\frac{\alpha'}{4} ({\rm tr}\ \Omega_{s,r}\wedge\Omega_{s,r} - {\rm tr}\ \Omega^A\wedge\Omega^A)} \\[8pt]
\zzz & = &\zzz \displaystyle{\frac{\alpha'}{4} \left( 4\,r^4\, \frac{4s^2+1}{s^4} - \frac{s^2}{2} \right)\, \omega^{1\bar{1}2\bar{2}}.}
\end{array}
\end{equation}
This implies that $\alpha'=\frac{8r^2(2s^2+1)}{8\,r^4(4s^2+1)-s^6}$.
Therefore, on the compact complex manifold $X^s$ we can take a balanced metric $F_{s,r}$ in~\eqref{new-balanced-family}
with $r \in \mathbb{R}$ large enough such that $r^4> \frac{s^6}{8(4s^2+1)}$, which ensures that $\alpha'>0$.
}
\end{rmrk}

\begin{rmrk}\label{other-example-no-open-negative}
{\rm
Notice that, in Remark~\ref{remark-Strominger-no-open}, taking $r$ small enough such that $0 < r^4 < \frac{s^6}{8(4s^2+1)}$,
one has on every compact complex manifold $X^s$ a solution to the Strominger system with $\alpha'<0$ in the anomaly cancellation equation.
Thus, using the holomorphic deformation of $X^s$ constructed in Theorem~\ref{Strominger-no-open},
we get a family of examples for which the Strominger property with $\alpha'<0$ is non-stable.
}
\end{rmrk}

\begin{rmrk}\label{remark-Strominger-no-closed}
{\rm
In the proof of Theorem~\ref{no-closed} we have considered on $X_\mathbf{t}$, $0<|\mathbf{t}|<1$, the trivial instanton $A$
in the anomaly cancellation equation. It is worth to remark that one can construct solutions with $\alpha'>0$ and non-flat instanton.
Indeed, following the proof of Theorem~\ref{negative-no-open-no-closed}, it is enough to take a balanced metric
$F_{\mathbf{t},r}$ in~\eqref{metric-disc-a}
with $r\in \mathbb{R}^*$ large enough such that $r^4> \frac{|\mathbf{t}|^4(1-|\mathbf{t}|^2)^4}{2(1+|\mathbf{t}|^2+2|\mathbf{t}|^4)}$,
which ensures $\alpha'>0$.
}
\end{rmrk}

\section*{Acknowledgments}
\noindent This work has been partially supported by the projects MINECO (Spain) MTM2014-58616-P,
Gobierno de Arag\'on/Fondo Social Europeo--Grupo Consolidado E15 Geometr\'{\i}a,
and by Fundaci\'on Bancaria Ibercaja--Fundaci\'on CAI--Universidad de Zaragoza, Programa de Estancias de Investigaci\'on,
Contract DFNI I02/4/12.12.2014 and  Contract 195/2016 with the
Sofia University "St.Kl.Ohridski". S.I. thanks the University of Zaragoza for the support during his visit to the Department of Mathematics, and
L.U. thanks the University of Sofia "St. Kl. Ohridski"
for the hospitality and financial support provided while visiting
the Faculty of Mathematics and Informatics.


\begin{thebibliography}{33}

\bibitem{AB} L. Alessandrini, G. Bassanelli, Small deformations of a class of compact non-K\"ahler manifolds,
\emph{Proc. Amer. Math. Soc.} {\bf 109} (1990), 1059--1062.

\bibitem{Anderson-G-S} L.B. Anderson, J. Gray, E. Sharpe,
Algebroids, heterotic moduli spaces and the Strominger system,
\emph{J. High Energy Physics} JHEP {\bf 07} (2014) 037.

\bibitem{A-Gar} B. Andreas, M. Garc\'{\i}a-Fern\'andez,
Solutions of the Strominger system via stable bundles on Calabi-Yau threefolds,
\emph{Commun. Math. Phys.} {\bf 315} (2012), 153--168.

\bibitem{AGarcia} B. Andreas, M. Garc\'{\i}a-Fern\'andez,
Note on solutions of the Strominger system from unitary representations of cocompact lattices of SL(2,$\mathbb{C}$),
\emph{Commun. Math. Phys.} {\bf 332} (2014), 1381--1383.

\bibitem{AFR} D. Angella, M.G. Franzini, F.A. Rossi, Degree of non-K\"ahlerianity for 6-dimensional nilmanifolds,
\emph{Manuscripta Math.} {\bf 148} (2015), no. 1--2, 177--211.

\bibitem{AU} D. Angella, L. Ugarte,
On small deformations of balanced manifolds,
arXiv:1502.07581v2 [math.DG].

\bibitem{BBDGE}  K. Becker, M. Becker, K. Dasgupta, P.S. Green, E. Sharpe,
Compactifications of heterotic strings on non-K\"ahler complex manifolds: II,
\emph{Nuclear Phys.} {\bf B} {\bf 678} (2004), 19--100.

\bibitem{BT} K. Becker and L.-S. Tseng, Heterotic flux compactifications and their moduli,
\emph{Nuclear Phys.} {\bf B 741} (2006) 162--179.

\bibitem{Berg} E.A. Bergshoeff, M. de Roo,
The quartic effective action of the heterotic string and supersymmetry,
\emph{Nuclear Phys.} {\bf B} {\bf 328} (1989), 439.

\bibitem{Bis} J.-M. Bismut, A local index theorem for non-K\"ahler
manifolds, \emph{Math. Ann.} {\bf 284} (1989), 681--699.

\bibitem{Bogomolov} F.A. Bogomolov,
Hamiltonian K\"ahler manifolds,
\emph{Dolk. Akad. Nauk SSSR} {\bf 243} (1978) 1101--1104.

\bibitem{CCDLMZ} G.L. Cardoso, G. Curio, G. Dall'Agata, D. Lust, P. Manousselis, G. Zoupanos,
Non-K\"aeler string back-grounds and their five torsion
classes, \emph{Nuclear Phys.} {\bf B} {\bf 652} (2003), 5--34.

\bibitem{COUV} M. Ceballos, A. Otal, L. Ugarte, R. Villacampa,
Invariant complex structures on 6-nilmanifolds: classification, Fr\"olicher spectral sequence
and special Hermitian metrics,  \emph{J. Geom. Anal.} {\bf 26} (2016), 252--286.

\bibitem{OsSv2}  X.C. de la Ossa, E. Hardy, E.E. Svanes, The Heterotic Superpotential and Moduli,
\emph{J. High Energy Physics} JHEP {\bf 01} (2016) 049.

\bibitem{Ossa-Svanes} X.C. de la Ossa, E.E. Svanes,
Holomorphic bundles and the moduli space of $N=1$ supersymmetric heterotic compactifications,
\emph{J. High Energy Physics} JHEP {\bf 10} (2014) 123.

\bibitem{OsSv1}  X.C. de la Ossa, E.E. Svanes,
Connections, Field Redefinitions and Heterotic Supergravity,
\emph{J. High Energy Physics} JHEP {\bf 12} (2014) 008.

\bibitem{Fei} T. Fei,
A construction of non-K\"ahler Calabi-Yau manifolds and new solutions to the Strominger system,
\emph{Adv. Math.} {\bf 302} (2016), 529--550.

\bibitem{FHP}  T. Fei, Z. Huang, S. Picard,
A construction of infinitely many solutions to the Strominger system,
arXiv:1703.10067 [math.DG].

\bibitem{Fei-Yau} T. Fei, S.-T. Yau,
Invariant solutions to the Strominger system on complex Lie groups and their quotients,
\emph{Commun. Math. Phys.} {\bf 338} (2015), 1183--1195.

\bibitem{FIUV09} M. Fern\'andez, S. Ivanov, L. Ugarte, R. Villacampa,
Non-K\"ahler heterotic string compactifications with non-zero fluxes and constant dilaton,
\emph{Commun. Math. Phys.} {\bf 288} (2009), 677--697.

\bibitem{FIUV2014} M. Fern\'andez, S. Ivanov, L. Ugarte, D. Vassilev,
Non-Kaehler heterotic string solutions with non-zero fluxes and non-constant dilaton,
\emph{J. High Energy Physics} JHEP {\bf 06} (2014) 073.

\bibitem{FG} A. Fino, G. Grantcharov,
Properties of manifolds with skew-symmetric torsion and special holonomy,
\emph{Adv. Math.} {\bf 189} (2004), 439--450.

\bibitem{Fu-Yau-alpha-negative}
J-X. Fu, S-T. Yau,
A Monge-Amp\`ere type equation motivated by string theory,
\emph{Commun. Anal. Geom.} {\bf 15} (2007), 29--76.

\bibitem{Fu-Yau-alpha-positive} J-X. Fu, S-T. Yau,
The theory of superstring with flux on non-K\"ahler manifolds and the complex Monge-Amp\`ere equation,
\emph{J. Diff. Geom.} {\bf 78} (2008), 369--428.

\bibitem{MGarcia} M. Garc\'{\i}a-Fern\'andez,
\emph{Lectures on the Strominger system},
arXiv:1609.02615 [math.DG].

\bibitem{Gar-Rubio-Tipler} M. Garc\'{\i}a-Fern\'andez, R. Rubio, C. Tipler,
Infinitesimal moduli for the Strominger system and generalized Killing spinors,
to appear in \emph{Math. Ann.}, doi:10.1007/s00208-016-1463-5.

\bibitem{GP} E. Goldstein, S. Prokushkin,
Geometric model for complex non-K\"aehler manifolds with SU(3) structure,
\emph{Commun. Math. Phys.} {\bf 251} (2004), 65--78.

\bibitem {Hul} C.M. Hull,
Compactifications of the heterotic superstring,
\emph{Physics Letters} {\bf B 178} (4):357-364, 1986.

\bibitem{LUV} A. Latorre, L. Ugarte, R. Villacampa,
On the Bott-Chern cohomology and balanced Hermitian nilmanifolds,
\emph{Internat. J. Math.} {\bf 25} (2014), no. 6, 1450057, 24 pp.

\bibitem{Li-Yau} J. Li, S-T. Yau,
The existence of supersymmetric string theory with torsion,
\emph{J. Diff. Geom.} {\bf 70} (2005), no. 1, 143--181.

\bibitem{MPPS} C. Maclaughlin, H. Pedersen, Y.S. Poon, S. Salamon,
Deformation of 2-step nilmanifolds with abelian complex structures,
\emph{J. Lond. Math. Soc.} {\bf 73} (2006), 173--193.

\bibitem{Mi} M.L. Michelsohn,
On the existence of special metrics in complex geometry,
\emph{Acta Math.} {\bf 149} (1982), no. 3-4, 261--295.

\bibitem{OUV1} A. Otal, L. Ugarte, R. Villacampa,
Invariant solutions to the Strominger system and the heterotic equations of motion,
to appear in \emph{Nuclear Phys.} {\bf B},
arXiv:1604.02851 [math.DG].

\bibitem{PPZ} D.H. Phong, S. Picard, X. Zhang,
The Fu-Yau equation with negative slope parameter,
to appear in \emph{Invent. Math.}, doi:10.1007/s00222-016-0715-z.

\bibitem{PPZ1} D. H. Phong, S. Picard, X. Zhang,
The anomaly flow and the Fu-Yau equation, arXiv:1610.02740 [math.DG].

\bibitem{Pop2} D. Popovici, Deformation openness and closedness of various classes
of compact complex manifolds; examples, \emph{Ann. Sc. Norm. Super. Pisa Cl. Sci. (5)}
{\bf 13} (2014), no. 2, 255--305.

\bibitem{R1} S. Rollenske, Lie-algebra Dolbeault-cohomology and small deformations of nilmanifolds,
\emph{J. London Math. Soc.} {\bf 79} (2009), 346--362.

\bibitem{Sal} S. Salamon,
Complex structures on nilpotent Lie algebras,
\emph{J. Pure Appl. Algebra} {\bf 157} (2001), 311--333.

\bibitem {Str} A. Strominger,
Superstrings with torsion,
\emph{Nuclear Phys.} {\bf B} {\bf 274} (1986), 253.

\bibitem {Tian} G. Tian,
Smoothness of the universal deformation space of compact Calabi-Yau manifolds and its Petersson-Weil metric,
Mathematical aspects of string theory (San Diego, Calif., 1986), 629--646,
Adv. Ser. Math. Phys., 1, World Sci. Publishing, Singapore, 1987.

\bibitem {Todorov} A.N. Todorov,
The Weil-Petersson geometry of the moduli space of SU($n=3$) (Calabi-Yau) manifolds. I,
\emph{Commun. Math. Phys.} {\bf 126} (1989), 325--346.

\bibitem{UV2} L. Ugarte, R. Villacampa, Balanced Hermitian geometry on 6-dimensional nilmanifolds,
\emph{Forum Math.} {\bf 27} (2015),
1025--1070.

\bibitem{Yau78} S.-T. Yau,
On the Ricci curvature of a compact K\"ahler manifold and the complex Monge-Amp\`ere
equation, I,
\emph{Comm. Pure Appl. Math.} {\bf 31} (1978), 339--411.

\end{thebibliography}
\end{document}